\documentclass[11pt]{amsart}
\usepackage{amssymb}
\usepackage{amsmath}
\usepackage{amsthm}
\usepackage[latin1]{inputenc}
\usepackage[T1]{fontenc}
\usepackage{hyperref}
\usepackage{color}
\usepackage[all]{xy}
\usepackage{enumerate}
\usepackage{rotating}
\usepackage{pstricks}
\usepackage{pstricks-add}

\newcommand{\N}{\mathbb{N}}
\newcommand{\Z}{\mathbb{Z}}
\newcommand{\R}{\mathbb{R}}
\newcommand{\C}{\mathbb{C}}
\newcommand{\quat}{\mathbb{H}}
\renewcommand{\S}{{\mathrm{S}}^{2N-1}}
\newcommand{\Hab}{\mathcal{H}^{\alpha,\beta}(\C^N)}
\newcommand{\Vll}{V_n^{l,l'}}

\newcommand{\GLNC}{\ensuremath{\mathrm{GL}(N,\C)}}

\newcommand{\SpnR}{\ensuremath{\mathrm{Sp}(n,\R)}}
\newcommand{\SpnC}{\ensuremath{\mathrm{Sp}(n,\C)}}
\newcommand{\HC}{\ensuremath{\mathrm{H}_{\C}^{2m+1}}}
\renewcommand{\Re}{\mathop{\rm{Re}}}
\newcommand{\scal}[2]{\langle #1,#2\rangle}

\theoremstyle{plain}
\newtheorem{theorem}{Theorem}
\newtheorem{lemma}{Lemma}
\newtheorem{proposition}{Proposition}

\theoremstyle{definition}
\newtheorem{definition}{Definition}

\newtheorem{remark}{Remark}

\begin{document}

\title[Degenerate principal series of $\mathrm{Sp}(n,\C)$]{On the degenerate principal series of complex symplectic groups}

\author{Pierre Clare}
\address{Pierre Clare\\ The Pennsylvania State University\\ Department of Mathematics\\ McAllister Building\\ University Park, PA - 16802}
\email{clare@math.psu.edu}
\thanks{This work was primarily supported by a JSPS postdoctoral fellowship. Support was also provided by the MAPMO and the Pennsylvania State University.}

\begin{abstract}
We apply techniques introduced by Clerc, Kobayashi, \O rsted and Pevzner to study the degenerate principal series of $\SpnC$. An explicit description of the $K$-types is provided and Knapp-Stein normalised operators are realised as symplectic Fourier transforms, and their $K$-spectrum explicitely computed. Reducibility phenomena are analysed in terms of $K$-types and eigenvalues of intertwining operators. We also construct a new model for these representations, in which Knapp-Stein intertwiners take an algebraic form.
\end{abstract}

\maketitle

\section*{Introduction}

\subsection*{Background and purpose}

Among the representations of reductive Lie groups, the so-called \textit{small representations} have received a lot of attention for several years. After the accomplishments of the 90's, mainly obtained through algebraic methods (see the introduction and the references in \cite{Alganalmin}), recent progress has stemmed from the developement of new techniques in geometric analysis. Leading work in this direction was the series \cite{KOadv03I,KOadv03II,KOadv03III}, dealing with the minimal representation of general indefinite orthogonal groups.

More recently, new techniques were introduced in this area by J.-L. Clerc, T. Kobayashi, B. \O rsted and M. Pevzner in relation with various problems in classical analysis \cite{BernsteinRez} and representation theory \cite{Kobayashi_small_GL}. The first of these works is devoted to the computation of certain triple integrals, and features the study of certain intertwining operators related to representations of the real symplectic group. More precisely, it is observed that some Knapp-Stein operators can be realised as Fourier transforms once normalised. This idea was already used in the work of A. Unterberger in the case of $\mathrm{SL}(2,\R)$, see \cite{Unterbergerlivre}. It was extended to $\mathrm{SL}(n,\R)$ in \cite{PevznerUnter} and to $\mathrm{Sp}(n,\R)$ for the first time in \cite{KOPUWeylcalc}. It is taken further in \cite{Kobayashi_small_GL} where, among many other results, a complete description of Knapp-Stein intertwiners between the degenerate principal series of $\mathrm{Sp}(n,\R)$ is carried out. This point of view allows to establish very precise statements, such as $K$-type formulas and  explicit computations of the $K$-spectrum of intertwining operators. Another feature of \cite{Kobayashi_small_GL} is the construction of what is named there the \textit{non-standard model} for degenerate principal series. As a first application, this apparently new picture allows to define Knapp-Stein operators by a rather simple algebraic formula.

The present article aims at adapting those techniques to the case of the complex symplectic group. Among the parabolically induced representations of $\SpnC$, those coming from maximal parabolic subgroups may be considered as the most degenerate, since functions on the corresponding flag manifold depend on the smallest possible set of parameters. These representations have first been investigated by K. I. Gross in \cite{Gross}.

Most of our arguments are rather directly inspired by the results in \cite{BernsteinRez} and \cite{Kobayashi_small_GL}, dealing with real symplectic groups. However, we tried to provide a self-contained and detailed presentation of these techniques while adapting them to the complex case. Moreover, the isotypic decomposition of the quaternionic orthogonal group action on square-integrable functions over the unit sphere, studied in Section~\ref{RCHsph}, turns out to be slightly more delicate than the real and complex ones. It is also worth noticing Theorem~\ref{Ktypethm} accounts more precisely for the reducibility phenomena than the original description in Gross' paper \cite{Gross}. In particular, the characterisations by $K$-types and eigenspaces of the algebraic intertwiners in the non-standard model both seem new.

Let us finally mention that the study of explicit Knapp-Stein intertwiners as geometric transforms and the computation of their $K$-spectrum such as the one provided in Proposition~\ref{spec} are a current object of interest. Indeed, analogous results were recently obtained in~\cite{PasquOlaf} for special linear groups, using techniques of \cite{Specgeneratop}.

\subsection*{Outline}
The article is organised as follows: general notations are fixed and elementary facts regarding degenerate principal series of the complex symplectic groups are stated in Section~\ref{setting}. In Section~\ref{FourierKnappStein}, we introduce certain Fourier transforms, establish some of their elementary properties and use them to normalise Knapp-Stein operators in Proposition~\ref{propFsympT}. In Section~\ref{Ktyp}, we study the branching law of the degenerate principal series representations of $\mathrm{Sp}(n,\C)$ with respect to the maximal compact subgroup $K=\mathrm{Sp}(n)$: we describe $K$-types in Proposition~\ref{Ktypeformula} and compute the eigenvalues of Knapp-Stein operators in Proposition~\ref{spec}. As a result, we are able to analyse the reducible elements in the degenerate principal series in terms of $K$-types and eigenspaces of te Knapp-Stein intertwiners in Theorem~\ref{Ktypethm}. Finally, Section~\ref{intertwiners} is devoted to the description of the \textit{non-standard} model of the degenerate principal series in the sense of \cite{Kobayashi_small_GL}. The main result of this section is the computation of the normalised Knapp-Stein operators in this picture: Theorem~\ref{algKS} establishes that the intertwiners are defined by an algebraic formula in this setting.

\section{Setting and notations}\label{setting}

\subsection{The complex symplectic group}

For any integer $p\geq1$, let $I_p$ be the identity matrix of size $p$ and let brackets denote the associate bilinear form on $\C^p$: \[\scal{X}{Y}=\sum_{k=1}^px_ky_k\] for $X=(x_1,\ldots,x_p)$ and $Y=(y_1,\ldots,y_p)$.

Throughout this article, $n$ shall be a fixed positive integer and $N=2n$, so that a vector $X$ in $\C^N\simeq\C^n\times\C^n$ can naturally be written $X=(X_1,X_2)$. The \emph{complex symplectic form} on $\C^N$ is defined by \[\omega_n(X,Y)=\scal{X_2}{Y_1}-\scal{X_1}{Y_2},\] that is $\omega_n(X,Y)=\scal{X}{JY}$ where \[J=\left[\begin{array}{c|c}0&-I_n\\\hline \;I_n\;&0\end{array}\right].\] Abusing notations, we will usually drop the subscript indicating the dimension and write $\omega\equiv\omega_n$ when no confusion may result.

By definition, the \emph{complex symplectic group} is the group of complex invertible matrices preserving $\omega$: \[\SpnC=\left\{g\in\GLNC\;\left|\;\forall X,Y\in\C^N,\right.\;\omega(gX,gY)=\omega(X,Y)\right\}.\]
Equivalently, $\SpnC$ is the subgroup of elements $g\in\GLNC$ subject to the relation $^{t}gJg=J$.

From now on, $G$ will denote the complex symplectic group defined above. Restricting the usual Cartan involution of $\GLNC$ yields a Cartan involution of $G$. As a consequence, $K=\mathrm{U}(N)\cap\SpnC$ is a maximal compact subgroup of $G$, also called the \emph{compact symplectic group} and denoted by $\mathrm{Sp}(n)$.

\subsection{Maximal parabolic subgroup of Heisenberg type}

Let us recall some facts regarding complex Heisenberg groups. Let $m=n-1$ and consider \[\HC=\left\{(s,X)\in\C\times\C^{2m}\right\}\] equipped with the product \[(s,X)(s',X')=\left(s+s'+\frac{1}{2}\omega(X,X')\,,\,X+X'\right),\] where $\omega\equiv\omega_m$ denotes the complex symplectic form on $\C^{2m}$.

The group $G$ acts naturally on $\C^N$ by linear applications, hence also on the complex projective space $\mathrm{P}^{N-1}\C$. The stabiliser in $G$ of a point in $\mathrm{P}^{N-1}\C$ is a maximal parabolic subgroup $P$ with Langlands decomposition \[P=MA\bar{N}\simeq\left(\C^\times.\mathrm{Sp}(m,\C)\right)\ltimes\HC.\]
Elements in the Cartan-stable Levi component $L=MA$ are of the form \[l(a,S)=\left[\begin{array}{cccc}a&0&0&0\\0&s_{11}&0&s_{12}\\0&0&a^{-1}&0\\0&s_{21}&0&s_{22}\end{array}\right]\] with $a\in\C^\times$ and $S=\left[\begin{array}{cc}s_{11}&s_{12}\\s_{21}&s_{22}\end{array}\right]\in\mathrm{Sp}(m,\C)$.

The Lie algebra $\mathfrak{g}$ of $G$ then admits a Gelfand-Naimark decomposition $\mathfrak{g}=\mathfrak{n}+\mathfrak{m}+\mathfrak{a}+\bar{\mathfrak{n}}$ and the analytic subgroup $N$ of $G$ with Lie algebra $\mathfrak{n}$ is another copy of $\HC$ which embeds in $G$ \textit{via} \begin{equation}\label{embed}\left(s,(X_1,X_2)\right)\longmapsto\left[\begin{array}{cccc}1&0&0&0\\X_1&I_m&0&0\\2s&^{t}X_2&1&-^{t}X_1\\X_2&0&0&I_m\end{array}\right]\end{equation}

\subsection{Degenerate principal series}\label{degeneratePS}

For $a\in\C^\times$, we denote $[a]=\frac{a}{|a|}$. Let $(\mu,\delta)\in\C\times\Z$. Such a couple defines a character $\chi_{\mu,\delta}$ of $P$ by \[\chi_{\mu,\delta}\left(l(a,S)\right)=|a|^\mu\left[a\right]^\delta.\] From now on, we assume $\mu$ to be purely imaginary, so that $\chi_{\mu,\delta}$ is unitary.

\begin{definition}
The induced representation \[\pi_{\mu,\delta}=\mathop{\rm{Ind}}\nolimits_{P}^{G}{\chi_{\mu,\delta}\otimes 1}\] is called a \emph{degenerate principal series representation} of $G$.
\end{definition}

These representations may be described in several ways. 

\subsubsection{Induced picture}

In this model, $\pi_{\mu,\delta}$ is realised on a space of square-integrable sections of the line bundle $G\times_{\chi_{\mu,\delta}}\C$ over the flag manifold $G/P\simeq\mathrm{P}^{N-1}\C$. A dense subspace of the carrying space in this picture is \[V_{\mu,\delta}^\infty = \left\{f\in C^\infty(\C^N\setminus\left\{0\right\})\,\left|\,\forall a\in\C^\times\,,\,f(a\,\cdot)=|a|^{-\mu-N}[a]^{-\delta}f\right.\right\}.\] By homogeneity, functions in $V_{\mu,\delta}^\infty$ are determined by their restriction to the unit sphere in $\C^N$, which we shall always identify to the $2N-1$-dimensional Euclidean sphere $\S$. The space $V_{\mu,\delta}$ is defined as the completion of $V_{\mu,\delta}^\infty$ with respect to the $L^2$-norm on $\S$, and $G$ acts by left multiplications.

\subsubsection{Compact picture}\label{compict}

Besides considering sections over the flag manifold $G/P$, one may restrict the induced picture to sections over $K/(L\cap K)$ where $L$ denotes the Levi component $\C^\times.\mathrm{Sp}(m,\C)$ of $P$. Considering the compact symplectic group as the orthogonal group of a quaternionic vector space gives the identification \begin{equation}\label{SpS}\S\simeq\mathrm{Sp}(n)/\mathrm{Sp}(m)\end{equation} so that \[K/(L\cap K)\simeq\mathrm{Sp}(n)/\mathrm{U}(1).\mathrm{Sp}(m)\simeq\S/\mathrm{U}(1)\] and $\pi_{\mu,\delta}$ is realised on \[L^2\left(\S\right)_\delta=\left\{f\in L^2\left(\S\right)\,\left|\,\forall \theta\in\R\,,\,f(e^{i\theta}\,\cdot)=e^{-i\delta\theta}f\right.\right\}.\]
The action of $G$ in this picture is slightly more complicated than in the induced one. However, its restriction to $K$ reduces to the regular action by left multiplication: $\pi_{\mu,\delta}^\mathrm{compact}(k)(f)=f(k^{-1}\cdot)$. More details about real, complex and quaternionic spheres and the isotypical decompositions of the associated $L^2$-spaces will appear in Section~\ref{Ktyp}, in order to analyse the $K$-types of the representations $\pi_{\mu,\delta}$ and determine the behaviour of the Knapp-Stein intertwiners on these $K$-types.

\subsubsection{Non-compact picture}\label{noncompict}

Another standard picture for principal series representations is obtained by restricting functions in the induced picture to $N$. More precisely, through the embedding~(\ref{embed}), any $f$ in $V_{\mu,\delta}^\infty$ gives a function on $\HC$ defined by \[(s,X_1,X_2)\longmapsto f(1,2s,X_1,X_2)\]and still denoted by $f$. It follows that $\pi_{\mu,\delta}$ is realised in $L^2\left(\HC\right)$.

\smallskip

In Section~\ref{intertwiners}, we introduce a new model for degenerate principal series representations and discuss its advantages in their study.

\section{Fourier transforms and Knapp-Stein integrals}\label{FourierKnappStein}

Various integral transforms will be used in relation to Knapp-Stein intertwining integrals. We shall define them on the space $\mathcal{S}(\C^N)$ of rapidly decreasing functions and extend them to the Schwartz space $\mathcal{S}'(\C^N)$ of tempered distributions by duality. The  \emph{complex Fourier transform} of $f\in\mathcal{S}(\C^N)$ is defined by \[\mathcal{F}_{\C^N}f(\xi)=\int_{\C^N}f(X)e^{-2i\pi\Re{\scal{X}{\xi}}}\,dX.\]

The factors in the product decomposition of $\C^N$ will be labelled $\C_i^n$ with $i\in\left\{1,2\right\}$ so that $\C^N=\C_1^n\times\C_2^n$ and the \textit{partial Fourier transform} with respect to the $i$-th variable will be denoted by $\mathcal{F}_{\C_i^n}$. Thus for $f\in\mathcal{S}(\C^N)$, \[\mathcal{F}_{\C_2^n}f(X_1,\xi_2)=\int_{\C^n}f(X_1,X_2)e^{-2i\pi\Re{\scal{X_2}{\xi_2}}}\,dX_2.\]

Finally, the \emph{complex symplectic Fourier transform} is defined on $\mathcal{S}(\C^N)$ by \[\mathcal{F}_{\mathrm{symp}}f(\xi)=\int_{\C^N}f(X)e^{-2i\pi\Re{\omega(X,\xi)}}\,dX,\] that is \[\mathcal{F}_{\mathrm{symp}}f(\xi)=\mathcal{F}_{\C^N}f(J\xi).\]

The following lemmas state elementary properties of the above transforms, to be used further.

\begin{lemma}\label{lemmaFC2}
Let $f\in V_{-\mu,-\delta}$ and $a\in\C^\times$. Then $\forall(X_1,\xi_2)\in\C^N$, \[\mathcal{F}_{\C_2^n}f(a X_1,a^{-1}\xi_2)=|a|^\mu[a]^\delta\mathcal{F}_{\C_2^n}f(X_1,\xi_2)\]
\end{lemma}

\begin{proof}
By definition, \begin{eqnarray*}\mathcal{F}_{\C_2^n}f(a X_1,a^{-1}\xi_2)&=&\int_{\C_2^n}f(aX_1,X_2)e^{-2i\pi\Re{\scal{X_2}{a^{-1}\xi_2}}}\,dX_2\\
&=&|a|^{2n}\int_{\C_2^n}f(aX_1,aX_2)e^{-2i\pi\Re{\scal{X_2}{\xi_2}}}\,dX_2\\
&=&|a|^{2n}|a|^{\mu-2n}[a]^\delta\mathcal{F}_{\C_2^n}f(X_1,\xi_2),\end{eqnarray*}hence the result.\end{proof}

\begin{lemma}\label{flip}
Let $f\in\mathcal{S}(\C^N)$. Then $\forall (u,v)\in\C_1^n\times\C_2^n\simeq\C^N$, \[\left(\mathcal{F}_{\C_2^n}\circ\mathcal{F}_{\mathrm{symp}}\circ\mathcal{F}_{\C_2^n}^{-1}\right)f(u,v)=f(v,u).\]
\end{lemma}

\begin{proof}
Let us compute:
\begin{align*}
\left(\mathcal{F}_{\C_2^n}\circ\mathcal{F}_{\mathrm{symp}}\circ\mathcal{F}_{\C_2^n}^{-1}\right)&f(u,v)\\
=\int_{\C_2^n\times\C_1^n\times\C_2^n\times\C_2^n}&f(u'',v''')e^{-2i\pi\Re{\left(\scal{v'}{v}+\scal{v''}{u}-\scal{u''}{v'}-\scal{v'''}{v''}\right)}}\,dv'''\,du''\,dv''\,dv'\\
=\int_{\C_2^n\times\C_1^n\times\C_2^n\times\C_2^n}&f(u'',v''')e^{-2i\pi\Re{\left(\scal{v'}{v-u''}+\scal{v''}{u-v'''}\right)}}\,dv'''\,du''\,dv''\,dv'\\
=\int_{\C_1^n\times\C_2^n}f(u'',v&''')\delta\left(v-u''\right)\delta\left(u-v'''\right)\,dv'''\,du'',
\end{align*}
hence the expected equality.
\end{proof}

Let us introduce some more notations: $\varepsilon$ denote the real matrix of size $2N$ defined by blocks as follows: \[\varepsilon = \left[\begin{array}{c|c}I_{N}&0\\\hline \;0\;&-I_{N}\end{array}\right],\] and if $f$ is a function on $\R^{2N}$, we denote by $f^\varepsilon$ the function $X\longmapsto f(\varepsilon X)$.

\begin{remark}\label{Feps}
Under the identification between $\R_1^N\times\R_2^N$ and $\C^N$ given by $(X_1,X_2)\longmapsto X_1 + iX_2$, the transformation $\varepsilon$ induces the complex conjugation. It follows that the complex Fourier transform can be seen as the transform $\mathcal{F}_\varepsilon$ defined on $\mathcal{S}(\R^{2N})$ by:
\[\mathcal{F}_\varepsilon\,f(\xi) = \int_{\R_1^{N}\times\R_2^N}f(X)e^{-2i\pi\left(\scal{X_1}{\xi_1}-\scal{X_2}{\xi_2}\right)}\,d(X_1,X_2),\]
so that  $\mathcal{F}_\varepsilon\,f = \left(\mathcal{F}_{\R^{2N}}\,f\right)^\varepsilon$. Indeed, if $\xi = (\xi_1,\xi_2)$ and $\zeta=\xi_1 + i\xi_2$, it is clear that \[\mathcal{F}_{\C^N}\,f(\zeta) = \mathcal{F}_\varepsilon\,f(\xi).\]
\end{remark}

\begin{definition}\label{plambda}
If $p$ is a function in $C_c^\infty(\S)$ and $\lambda$ is a complex number, we denote $p_\lambda$ the function defined by extending $p$ to $\R^{2N}\setminus\left\{0\right\}$ by \[p_\lambda(rX)=r^{\lambda}p(X)\] for $r>0$ and $X\in\S$.
\end{definition}

Following \cite{BernsteinRez}, we also consider the meromorphic function in the complex variable $\lambda$ defined by \[B_{2N}(\lambda,k)=\pi^{-\lambda-N}i^{-k}\frac{\Gamma\left(N + \frac{k+\lambda}{2}\right)}{\Gamma\left(\frac{k-\lambda}{2}\right)}.\] Then, denoting $\mathcal{H}^k(\R^{2N})$ the space of harmonic homogeneous polynomials of degree $k$ over $\R^{2N}$, the following holds:

\begin{lemma}\label{Fepslambda}
Let $p$ be the restriction to $\S$ of a polynomial in $\mathcal{H}^k(\R^{2N})$. The following identity between distributions on $\R^{2N}$ depending meromorphically on $\lambda$ holds: \begin{equation}\mathcal{F}_\varepsilon\,p_\lambda = B_{2N}(\lambda,k)p_{-\lambda-2N}^\varepsilon.\label{eqFepslambda}\tag{$\dag$}\end{equation}
\end{lemma}

\begin{proof}
Following the lines of the proof of \cite[Lemma 2.7]{BernsteinRez}, it is enough to prove that \begin{equation}\scal{\mathcal{F}_\varepsilon\,p_\lambda}{gq}=B_{2N}(\lambda,l)\scal{p_{-\lambda-2N}^\varepsilon}{gq}\label{eq}\end{equation} for $g\in C_c^\infty(\R_+)$ and $q\in\mathcal{H}^l(\R^{2N})$ in the domain \[-2N<\Re \lambda<-\left(N+\frac{1}{2}\right)\] to ensure that (\ref{eqFepslambda}) holds on $\R^{2N}$. Local integrability of $p_\lambda$ and $p_{-\lambda-2N}^\varepsilon$ follows from the choice of the domain. Denoting by $J_\mu(\nu)$ the Bessel function of the first kind, the Bochner identity directly implies that \[\int_{\S}q(\omega)e^{-i\scal{\omega}{\varepsilon\eta}}\,d\sigma(\omega) = (2\pi)^Ni^{-l}\nu^{1-N}J_{l+N-1}(\nu)q(\varepsilon\eta)\]
so that \begin{eqnarray*}\mathcal{F}_\varepsilon\,gq(r\omega)&=&\int_0^{+\infty}\int_{\S} g(s)q(\omega')e^{-2i\pi rs\scal{\omega'}{\varepsilon\omega}}s^{2N-1}\,ds\,d\sigma(\omega')\\&=&2\pi i^{-l}r^{1-N}q(\varepsilon\omega)\int_0^{+\infty}s^N g(s) J_{l+N-1}(2\pi rs)\,ds.\end{eqnarray*}
It follows that \begin{eqnarray*}\scal{p_\lambda}{\mathcal{F}_\varepsilon(gq)}&=&\int_0^{+\infty}\int_{\S} r^\lambda p(\omega)\mathcal{F}_\varepsilon\,gq(r\omega)r^{2N-1}\,dr\,d\sigma(\omega)\\&=&\int_0^{+\infty}\int_{\S}\left(\int_0^{+\infty}I(r,s)\,ds\right)p(\omega)q(\varepsilon\omega)\,d\sigma(\omega)\,dr\end{eqnarray*}
where we set \[I(r,s) = 2\pi i^{-l}r^{\lambda+N}s^N g(s) J_{l+N-1}(2\pi rs).\]
The proof of Claim 2.9 in \cite{BernsteinRez} ensures that $I$ belongs to $L^1\left(\R_+\times\R_+,dr\,ds\right)$ and that \[\int_0^{+\infty}I(r,s)\,dr = B_{2N}(\lambda,l)g(s)s^{-\lambda-1}.\]
As a consequence, \begin{eqnarray*}\scal{p_\lambda}{\mathcal{F}_\varepsilon(gq)}&=&\int_{\S} p(\omega)q(\varepsilon\omega)\,d\sigma(\omega)\int_0^{+\infty}\int_0^{+\infty}I(r,s)\,dr\,ds\\&=&B_{2N}(\lambda,l)\int_{\S} p(\omega)q(\varepsilon\omega)\,d\sigma(\omega)\int_0^{+\infty}g(s)s^{-\lambda-1}\,ds,\end{eqnarray*}
which implies that \[\scal{p_\lambda}{\mathcal{F}_\varepsilon(gq)} = B_{2N}(\lambda,l)\scal{p_{-\lambda-2N}}{gq^\varepsilon},\] thus proving (\ref{eq}) and the lemma.
\end{proof}

\begin{remark}\label{Fsympintertwines} The link between Fourier transforms and Knapp-Stein operators relies on the observation that $\mathcal{F}_{\mathrm{symp}}$ provides a unitary equivalence between $\pi_{\mu,\delta}$ and $\pi_{-\mu,-\delta}$. Indeed, denoting $f_a=f(a\,\cdot)$ for $f\in\mathcal{S}(\C^N)$ and $a\in\C^\times$, a single change of variables leads to \[\left(\mathcal{F}_{\mathrm{symp}}f\right)_a = |a|^{-2N}\mathcal{F}_{\mathrm{symp}}\left(f_{a^{-1}}\right).\]
If moreover $f\in V_{-\mu,-\delta}$ then, by linearity, \[\left(\mathcal{F}_{\mathrm{symp}}f\right)_a = |a|^{-2N}|a|^{-\mu+N}[a]^{-\delta}\mathcal{F}_{\mathrm{symp}}f,\] that is \[\mathcal{F}_{\mathrm{symp}}:V_{-\mu,-\delta}\longrightarrow V_{\mu,\delta}.\] Finally, for any $g\in G$, \[\mathcal{F}_{\mathrm{symp}}\left(\pi_{-\mu,-\delta}(g)f\right)(\xi) = \int_{\C^N}f(X)e^{-2i\pi\Re{ \omega(gX,\xi)}}\,dX\] and since $\omega$ is preserved by $G$, it follows that $\mathcal{F}_{\mathrm{symp}}\pi_{-\mu,-\delta}(g)=\pi_{\mu,\delta}(g)\mathcal{F}_{\mathrm{symp}}$.
\end{remark}

Next we introduce the Knapp-Stein operators which will be related to $\mathcal{F}_{\mathrm{symp}}$ after normalisation.

\begin{definition}
The operator $\mathcal{T}_{\mu,\delta}:V_{-\mu,-\delta}\longrightarrow V_{\mu,\delta}$ obtained by meromorphic continuation with respect to $\mu$ of the integral \[\mathcal{T}_{\mu,\delta}f(Y)=\int_{\S}f(X)\left|\Re{\omega(X,Y)}\right|^{-\mu-N}\left[\Re{\omega(X,Y)}\right]^{-\delta}\,d\sigma(X),\] where $d\sigma$ is the Euclidean measure on the unit sphere of $\R^{2N}$, is called the \emph{Knapp-Stein operator}, as introduced in \cite{KS1}, associated to the parameter $(\mu,\delta)\in\C\times\Z$.
\end{definition}

\begin{remark}
The kernel defining the operator in the above definition depends only on $\mu$ and the class of $\delta$ in $\Z/2\Z$.
\end{remark}

The operators $\mathcal{T}_{\mu,\delta}$ enjoy the intertwining property, yet as always in Knapp-Stein theory \cite{KS1,KS2}, they are not unitary at first. However, normalisation may be obtained by using the symplectic Fourier transform. More precisely, let \[C_N(\mu,\delta)=\left\{\begin{array}{cc}2\pi^{\mu+N-\frac{1}{2}}\dfrac{\Gamma\left(\frac{1-\mu-N}{2}\right)}{\Gamma\left(\frac{\mu+N}{2}\right)}&\text{if $\delta$ is even}\\-2i\pi^{\mu+N-\frac{1}{2}}\dfrac{\Gamma\left(\frac{2-\mu-N}{2}\right)}{\Gamma\left(\frac{1+\mu+N}{2}\right)}&\text{if $\delta$ is odd.}\end{array}\right.\] Then the followings holds:

\begin{proposition}\label{propFsympT}
The \emph{normalised Knapp-Stein operator} associated to $(\mu,\delta)$ is defined by \[\widetilde{\mathcal{T}}_{\mu,\delta} = \frac{1}{C_N(\mu,\delta)}T_{\mu,\delta}.\]
As a meromorphic extension in the complex variable $\mu$, it satisfies \[\widetilde{\mathcal{T}}_{\mu,\delta} = \mathcal{F}_{\mathrm{symp}}\left|_{V_{-\mu,-\delta}}.\right.\label{FsympT}\tag{$\ddag$}\]
As a consequence, $\widetilde{\mathcal{T}}_{\mu,\delta}$ yields a unitary equivalence between $\pi_{-\mu,-\delta}$ and $\pi_{\mu,\delta}$.
\end{proposition}

\begin{proof}Following \cite[Prop. 2.13]{BernsteinRez}, we identify $\C^N\setminus\left\{0\right\}$ to $\R_+^*\times\S$ by using spherical coordinates $(r,X)$. Then any function in $V^\infty_{-\mu,-\delta}$ is of the form \[h_{\mu-N}(rX)=r^{-\mu-N}h(X)\] for some $h\in C^\infty\left(\S\right)$ satisfying $h(e^{i\theta}X)=e^{i\delta\theta}h(X)$ for any $X\in\S$ and $\theta\in\R$, that is $h$ is a $C^\infty$ function in $L^2(\S)_{-\delta}$. The spherical component $h$ being fixed, it is enough to prove \[\mathcal{T}_{\mu,\delta}h_{\mu-N} = C_N(\mu,\delta)\mathcal{F}_{\mathrm{symp}}\left|_{V_{-\mu,-\delta}}\right.h_{\mu-N}\] for the parameter $\mu$ in a non-empty open domain. We shall work on the set defined by $\Re{\mu}>-N$. On this half-plane, $h_{\mu-N}$ is locally integrable and, defining $h_{\varepsilon,\mu-N}$ by \[h_{\varepsilon,\mu-N}(rX)=e^{-2\pi r}h_{\mu-N}(rX),\] one has $\lim\limits_{\varepsilon\to0^+}h_{\varepsilon,\mu-N} = h_{\mu-N}$ in $\mathcal{S}(\C^N)$, so that \[\mathcal{F}_{\mathrm{symp}}h_{\mu-N} = \lim_{\varepsilon\to0^+}\mathcal{F}_{\mathrm{symp}}h_{\varepsilon,\mu-N}.\]

\begin{align*}
\mathcal{F}_{\mathrm{symp}}h_{\varepsilon,\mu-N}(&sY)\\
& = \int_0^{+\infty}\int_{\S}e^{-2\pi r}r^{\mu-N}e^{-2i\pi rs\Re{\omega(X,Y)}}h(X)\,d\sigma(X)r^{2N-1}dr\\
& = \int_{\S}\mathcal{F}_\R\left(r_+^{\mu+N-1}\right)\left(s\Re{\omega(X,Y)}-i\varepsilon\right)h(X)\,d\sigma(X)
\end{align*}
where $\mathcal{F}_\R$ denotes the usual Fourier transform on the real line. Then, using classical formulas relative to Fourier transforms of homogeneous distributions (see~\cite{GelfandShilov}), one has \[\mathcal{F}_{\mathrm{symp}}h_{\varepsilon,\mu-N}(sY) = \frac{\Gamma(\mu+N)e^{-i\frac{\pi}{2}(\mu+N)}}{(2\pi)^{\mu+N}}\int_{\S}\left(s\Re{\omega(X,Y)}-i\varepsilon\right)^{-\mu-N}h(X)\,d\sigma(X)\]
and letting $\varepsilon\to0^+$ yields
\[\mathcal{F}_{\mathrm{symp}}h_{\mu-N}(sY) = \frac{\Gamma(\mu+N)e^{-i\frac{\pi}{2}(\mu+N)}}{(2\pi s)^{\mu+N}}\int_{\S}\left(\Re{\omega(X,Y)}-i0\right)^{-\mu-N}h(X)\,d\sigma(X).\]
 
Recall that
\begin{align*}
(x-i0)&^{-\mu-N}\\
&= e^{i\frac{\pi}{2}(\mu+N)}\left(\cos\left(\frac{\pi(\mu+N)}{2}\right)|x|^{-\mu-N} - i\sin\left(\frac{\pi(\mu+N)}{2}\right)|x|^{-\mu-N}\mathrm{sign}(x)\right)\\
&= \pi e^{i\frac{\pi}{2}(\mu+N)}\left(\frac{|x|^{-\mu-N}}{\Gamma\left(\frac{1+\mu+N}{2}\right)\Gamma\left(\frac{1-\mu-N}{2}\right)} - i\frac{|x|^{-\mu-N}\mathrm{sign}(x)}{\Gamma\left(\frac{\mu+N}{2}\right)\Gamma\left(\frac{2-\mu-N}{2}\right)}\right),
\end{align*}
using Euler's formula. The duplication formula satisfied by $\Gamma$ implies that: \[\Gamma\left(\mu+N\right)=2^{\mu+N-1}\pi^{-\frac{1}{2}}\Gamma\left(\frac{\mu+N}{2}\right)\Gamma\left(\frac{1+\mu+N}{2}\right).\]

Finally, we notice that $h$ is even if $h_{\mu-N}$ belongs to $V_{-\mu,-\delta}$ with $\delta\in2\Z$ and odd otherwise, so that $\mathcal{F}_{\mathrm{symp}}h_{\mu-N}(sY)$ is equal to \[\frac{\pi^{-\mu-N+\frac{1}{2}}}{2}\frac{\Gamma\left(\frac{\mu+N}{2}\right)}{\Gamma\left(\frac{1-\mu-N}{2}\right)}s^{-\mu-N}\int_{\S}\left|\Re{\omega(X,Y)}\right|^{-\mu-N}h(X)\,d\sigma(X)\]

in the even case, while in the odd case it amounts to

\[\frac{\pi^{-\mu-N+\frac{1}{2}}}{-2i}\frac{\Gamma\left(\frac{1+\mu+N}{2}\right)}{\Gamma\left(\frac{2-\mu-N}{2}\right)}s^{-\mu-N}\int_{\S}\left|\Re{\omega(X,Y)}\right|^{-\mu-N}\mathrm{sign}(\Re{\omega(X,Y)})h(X)\,d\sigma(X),\] thus proving~(\ref{FsympT}). The last statement then follows from Remark~\ref{Fsympintertwines}.
\end{proof}

For future reference, the main Fourier transforms of use in what follows are listed below.

\fbox{\begin{minipage}{125mm}
\begin{center}\textsc{Summary of integral transforms}\end{center}

\medskip

\emph{Complex Fourier transform}:

\[\mathcal{F}_{\C^N}f(\xi)=\int_{\C^N}f(X)e^{-2i\pi\Re{\scal{X}{\xi}}}\,dX\]

\medskip

\emph{Complex symplectic Fourier transform}: \quad $\mathcal{F}_{\mathrm{symp}}f(\xi)=\mathcal{F}_{\C^N}f(J\xi)$

\[\mathcal{F}_{\mathrm{symp}}f(\xi)=\int_{\C^N}f(X)e^{-2i\pi\Re{\omega(X,\xi)}}\,dX\]

\medskip

\emph{Partial Fourier transform}: on $\C^N\simeq\C_1^n\times\C_2^n$

\[\mathcal{F}_{\C_2^n}f(X_1,\xi_2)=\int_{\C^n}f(X_1,X_2)e^{-2i\pi\Re{\scal{X_2}{\xi_2}}}\,dX_2\]

\medskip

\emph{Real conjugate Fourier transform}: $\mathcal{F}_\varepsilon\,f(\xi) = \mathcal{F}_{\R^{2N}}f(\varepsilon\xi)$

\[\mathcal{F}_\varepsilon\,f(\xi_1,\xi_2) = \mathcal{F}_{\C^N}\,f(\xi_1+i\xi_2)\]

\[\mathcal{F}_\varepsilon\,f(\xi) = \int_{\R_1^{N}\times\R_2^N}f(X)e^{-2i\pi\left(\scal{X_1}{\xi_1}-\scal{X_2}{\xi_2}\right)}\,d(X_1,X_2)\]
\end{minipage}}

\section{\texorpdfstring{Restriction to $\mathrm{Sp}(n)$}{Restriction to Sp(n)}}\label{Ktyp}

The determination of a $K$-type formula for $\pi_{\mu,\delta}$ relies on some known facts regarding the representation theory of orthogonal groups over $\R$, $\C$ and $\quat$. More precisely it will involve the isotypical decompositions of square-integrable functions over the Euclidean unit sphere in $\R^{4n}\simeq\C^{2n}\simeq\quat^n$.

\subsection{Real, complex and quaternionic spherical harmonics}\label{RCHsph}

Let us fix the following classical identifications: \begin{equation}\label{RCH}\quat^n\simeq\left(\C^n+j\C^n\right)\simeq\left(\left(\R^n+i\R^n\right)+j\left(\R^n+i\R^n\right)\right)\end{equation} where $i$, $j$ and $k=ij$ denote the standard quaternion units. Then the unit spheres $S_1(\cdot)$ of those isometric vector spaces all identify to $\S$ and carry compatible left actions of the corresponding orthogonal groups as follows:

\begin{equation}\label{groupesortho}\begin{array}{ccccl}
\mathrm{Sp}(n)&\curvearrowright&S_1(\quat^n)&\curvearrowleft&\mathrm{Sp}(1)\simeq\mathrm{SU}(2)\\
\cap&&\begin{turn}{90}$\simeq$\end{turn}&&\:\cup\\
\mathrm{U}(2n)&\curvearrowright&S_1(\C^{2n})&\curvearrowleft&\mathrm{U}(1)\\
\cap&&\begin{sideways}$\simeq$\end{sideways}&&\:\cup\\
\mathrm{O}(4n)&\curvearrowright&S_1(\R^{4n})&\curvearrowleft&\left\{\pm1\right\}\\
&&\begin{sideways}$\simeq$\end{sideways}&&\\
&&\S&&\\
\end{array}\end{equation}

The right column displays the right actions of scalars of norm $1$. As an identification between $\mathrm{Sp}(1)$ and $\mathrm{SU}(2)$, we fix the one given by \begin{equation}\label{Sp1SU2}i\mapsto\left[\begin{array}{cc}i&0\\0&-i\end{array}\right]\quad,\quad j\mapsto\left[\begin{array}{cc}0&1\\-1&0\end{array}\right]\quad,\quad k\mapsto\left[\begin{array}{cc}0&i\\i&0\end{array}\right],\end{equation} so that $\mathrm{U}(1)$ naturally appears as a Cartan subgroup of $\mathrm{Sp}(1)$ \textit{via} the map \begin{equation}\label{CartanU1}e^{i\theta}\mapsto\left[\begin{array}{cc}e^{i\theta}&0\\0&e^{-i\theta}\end{array}\right].\end{equation}
Square integrable functions on $\S$ decompose with respect to the characters of $\left\{\pm1\right\}$ as even and odd, while the component corresponding to $\delta\in\Z\simeq\widehat{\mathrm{U}(1)}$ is the space $L^2\left(\S\right)_\delta$ introduced in Section~\ref{compict}. No ambiguity arises from the fact that the action was defined on the left there, since $\mathrm{U}(1)$ is abelian.

Let us now turn to the decomposition of $L^2\left(\S\right)$ into irreducible representations of $\mathrm{O}(4n)$ and $\mathrm{SU}(2n)$, that is the classical theory of spherical harmonics on real and complex vector spaces. More details may be found in \cite[Section 2.1]{BernsteinRez}.

As in Section~\ref{FourierKnappStein} we denote by $\mathcal{H}^k(\R^{2N})$ the vector space of harmonic homogeneous polynomials on $\R^{2N}$ of degree $k\in\N$. It is also useful to consider the space $\Hab$ of harmonic polynomials of the complex variable and its conjugate, homogeneous of degree $\alpha$ in $Z\in\C^N$ and of degree $\beta$ in $\bar{Z}$. Under the identifications of~(\ref{RCH}), there is a natural isomorphism

\[\mathcal{H}^k(\R^{2N})\simeq\bigoplus_{\alpha+\beta=k}\Hab.\]

Restricting functions to the sphere provides a complete orthogonal basis, hence a discrete sum decomposition of $L^2\left(\S\right)$ into irreducible components of the left actions of $\mathrm{O}(4n)$ and $\mathrm{U}(2n)$ in~(\ref{groupesortho}), namely

\begin{equation}\label{isotypRC}L^2\left(\S\right)\simeq\sideset{}{^\oplus}\sum_{k\geq0}\left.\mathcal{H}^k(\R^{2N})\right|_{\S}\simeq\sideset{}{^\oplus}\sum_{k\geq0}\bigoplus_{\alpha+\beta=k}\left.\Hab\right|_{\S}.\end{equation}

Taking into account the right actions in~(\ref{groupesortho}), one can refine~(\ref{isotypRC}) as \begin{align}L^2\left(\S\right)_{\mathrm{even}}\simeq\sideset{}{^\oplus}\sum_{k\in 2\N}\left.\mathcal{H}^k(\R^{2N})\right|_{\S}\\
L^2\left(\S\right)_{\mathrm{odd}}\simeq\sideset{}{^\oplus}\sum_{k\in 2\N+1}\left.\mathcal{H}^k(\R^{2N})\right|_{\S}\end{align}
in the real case and
\begin{equation}L^2\left(\S\right)_\delta\simeq\sideset{}{^\oplus}\sum_{\beta-\alpha=\delta}\left.\Hab\right|_{\S}\end{equation} in the complex case.

\begin{remark}The above discussion of the isotypical decomposition of $L^2\left(\S\right)$ with respect to the action of $\mathrm{U}(2n)\times\mathrm{U}(1)$ (resp. $\mathrm{O}(4n)\times\Z_2$) only involves the irreducible representations of $\mathrm{U}(2n)$ (resp. $\mathrm{O}(4n)$) because the commutativity of $\mathrm{U}(1)$ (resp. $\Z_2$) implies that these groups have 1-dimensional irreducible unitary representations, which is not the case of $\mathrm{Sp}(1)$.\end{remark}

In order to proceed with the same analysis over quaternions and write down the analogue of (\ref{isotypRC}) corresponding to the action of  $\mathrm{Sp}(n)\times\mathrm{Sp}(1)$ on $L^2\left(\S\right)$, some additional notations are needed. Following~\cite[Sections 5 and 6]{HoweTan}, we denote by $\Vll$ the unique unitary irreducible representation of $\mathrm{Sp}(n)$ corresponding to the highest weight $(l,l',0,\ldots,0)$ where $l$ and $l'$ are integers satisfying $l\geq l'\geq0$. Similarly, $V_1^j$ denotes the irreducible $j+1$-dimensional representation of $\mathrm{Sp}(1)\simeq\mathrm{SU}(2)$. 

Since $\mathrm{U}(1)$ naturally embeds into $\mathrm{SU}(2)$ \textit{via}~(\ref{CartanU1}), this representation decomposes according to the characters of the circle. More precisely, if $\C_\delta$ denotes the space of the character $z\mapsto z^\delta$ of $\mathrm{U}(1)$, for $\delta\in\Z$, then
\begin{equation}\label{VllCdelta}V_1^j\simeq\bigoplus_{\substack{|\delta|\leq j\\\delta\equiv j[2]}}\C_\delta.\end{equation}
Now identifying $\S$ to $S_1\left(\quat^n\right)$, the isotypic decomposition of $L^2(\S)$ with respect to action of $\mathrm{Sp}(n)\times\mathrm{Sp}(1)$ defined by the first line of (\ref{groupesortho}) is given in~\cite{HoweTan} by: \[L^2(\S)\simeq\sideset{}{^\oplus}\sum_{l\geq l'\geq0} \Vll\otimes V_1^{l-l'}.\] Together with~(\ref{VllCdelta}), this decomposition gives \begin{equation}\label{L2SVll}L^2(\S)\simeq\sideset{}{^\oplus}\sum_{\left(\delta,(l,l')\right)\in\Z\times\N^2\,,\,\left\{\substack{l-l'\geq |\delta|\\l-l'\equiv\delta[2]}\right.}\Vll\otimes\C_\delta\end{equation} under the action of $\mathrm{Sp}(n)\times\mathrm{U}(1)$.

\subsection{\texorpdfstring{The branching law $\SpnC\downarrow\mathrm{Sp}(n)$}{The branching law Sp(n,C) | Sp(n)}}

The above discussion leads to the determination of the branching law $\mathrm{Sp}(n,\C)\downarrow\mathrm{Sp}(n)$ of $\pi_{\mu,\delta}$ seen in the compact picture.

\begin{proposition}[$K$-type formula]\label{Ktypeformula}
The restriction of $\pi_{\mu,\delta}$ to the maximal compact subgroup $\mathrm{Sp}(n)$ decomposes into irreducible components as follows:
\[\left.\pi_{\mu,\delta}\right|_{\mathrm{Sp}(n)}\simeq\sideset{}{^\oplus}\sum_{\substack{l-l'\geq|\delta|\\l-l'\equiv\delta[2]}} \Vll.\]
Every $K$-type $\Vll$ occurs at most once in this decomposition.
\end{proposition}

\begin{proof}
The discussion of Paragraph~\ref{compict} shows how $\pi_{\mu,\delta}$ can be realised on $L^2\left(\S\right)_\delta$. In this picture, the restriction to $K$ of the action coincides with the natural representation of $\mathrm{Sp}(n)$ on functions over the unit sphere of $\quat^n$. Subsequently, the first statement reduces to fixing $\delta$ in~(\ref{L2SVll}). The fact that this decomposition is multiplicity-free relies on the observation that although a summand $\C_\delta$ appears in $V_1^{l-l'}$ for various values of $l-l'$, these do not involve the same space $\Vll$ more than once.
\end{proof}

\subsection{\texorpdfstring{Action of the Knapp-Stein operators on the $K$-types}{Action of the Knapp-Stein operators on the K-types}}

The last result of this section describes the behaviour of the Knapp-Stein intertwiners on each $K$-type. Proposition~\ref{propFsympT} proves that  $\widetilde{\mathcal{T}}_{\mu,\delta}$ intertwines $\pi_{\mu,\delta}$ and $\pi_{-\mu,-\delta}$. As a consequence of the formula in Proposition~\ref{Ktypeformula}, these representations have the same $K$-types $\Vll$. Taking into account the right action of $\mathrm{U}(1)$ seen a Cartan subgroup of $\mathrm{Sp}(1)$ \textit{via}~(\ref{groupesortho}) and the corresponding isotypic decomposition~(\ref{L2SVll}) of $L^2\left(\S\right)$, one is led to study the restriction \[\begin{array}{rccc}\mathcal{T}_{\mu,\delta}^{l,l'}:&\Vll\otimes\C_\delta&\longrightarrow&\Vll\otimes\C_{-\delta}\\&\cap&&\cap\\&\Vll\otimes V_1^{l-l'}&&\Vll\otimes V_1^{l-l'}\end{array}.\]

In order to specifiy how $\mathcal{T}_{\mu,\delta}$ restricts to an operator of $\Vll$, it is necessary to fix an identification between $\Vll\otimes\C_\delta$ and $\Vll\otimes\C_{-\delta}$. Using the isomorphism~(\ref{Sp1SU2}) between $\mathrm{Sp}(1)$ and $\mathrm{SU}(2)$, 
it is done by choosing a non-trivial element $w$ in the Weyl group $W(\mathrm{SU}(2):\mathrm{U}(1))$ and letting it act by conjugation on $\mathrm{SU}(2)$. Since such an action inverts the elements in the torus $\mathrm{U(1)}$, it provides an isomorphism $\iota_w$ from $\Vll$ to itself, that exchanges the summands $\C_\delta$ and $\C_{-\delta}$ appearing in~(\ref{VllCdelta}):
\begin{equation}\label{iotaw}\begin{array}{rccc}\iota_w:&V_1^{l-l'}&\stackrel{\sim}{\longrightarrow}&V_1^{l-l'}\\&\cup&&\cup\\&\C_\delta&\stackrel{\sim}{\longrightarrow}&\C_{-\delta}\end{array}.\end{equation}

The restricted intertwining operator under this identification will be denoted by $\mathcal{T}_{\mu,\delta}^{l,l',w}$: \[\mathcal{T}_{\mu,\delta}^{l,l',w}=\mathcal{T}_{\mu,\delta}^{l,l'}\otimes\iota_w.\]

From now on, $w$ will be the class modulo $\mathrm{U}(1)$ of $\left[\begin{array}{cc}0&1\\-1&0\end{array}\right]$. Since this matrix of $\mathrm{SU}(2)$ normalises $\mathrm{U}(1)$, it defines a Weyl element, hence fixes the above definition of $\mathcal{T}_{\mu,\delta}^{l,l',w}$.

\begin{proposition}\label{spec}
Let $\delta\neq0$. For $l,l'\in\N$ such that $l-l'\geq|\delta|$ and $l-l'\equiv\delta[2]$, the restriction $\mathcal{T}_{\mu,\delta}^{l,l',w}$ of the normalised Knapp-Stein intertwiner $\widetilde{\mathcal{T}}_{\mu,\delta}$ acts on $\Vll$ as the scalar \[\pi^{-\mu}(-i)^{-(l+l')}\frac{\Gamma\left(n+\frac{l+l'+\mu}{2}\right)}{\Gamma\left(n+\frac{l+l'-\mu}{2}\right)}.\]
\end{proposition}

\begin{proof}
Let $p\in\Vll\otimes\C_\delta\subset\Vll\otimes V_1^{l-l'}$. By compatibility of the isotypic decompositions~(\ref{isotypRC}) and ~(\ref{L2SVll}), $p$ can be seen as the restriction to $\S$ of a polynomial in $\mathcal{H}^{l+l'}(\R^{2N})$. In view of Proposition~\ref{propFsympT} and Lemma~\ref{Fepslambda}, the operator $\widetilde{\mathcal{T}}_{\mu,\delta}$ maps $p_{\mu-N}$ to $B_{2N}(\mu-N,l+l')p_{-\mu-N}(J\varepsilon\,\cdot)$. Under the identifications~(\ref{RCH}), applying $J\varepsilon$ to vectors in $\quat^n$ from the left amounts to multiplying them by the quaternionic unit $j$ from the right. It follows that \[\mathcal{T}_{\mu,\delta}^{l,l',w}p=B_{2N}(\mu-N,l+l')\iota_w\left(p.j\right)\] where $j\in\mathrm{Sp}(1)$ acts \textit{via} $V_1^{l-l'}$. Identifying $\mathrm{Sp}(1)$ to $\mathrm{SU}(2)$ by~(\ref{Sp1SU2}) again, we realise $V_1^{l-l'}$ as the classical representation of $\mathrm{SU}(2)$ on homogeneous polynomials of degree $l-l'$ in two variables $x$ and $y$, denoted by $\Phi^{l-l'}[x,y]$. Since $j$ and $w$ are both represented by the matrix $\left[\begin{array}{cc}0&1\\-1&0\end{array}\right]$, one has \[\iota_w\left(p.j\right)=p.(-I_2)=(-1)^\delta p,\] hence the conclusion: $\mathcal{T}_{\mu,\delta}^{l,l',w}\,p=(-1)^\delta B_{2N}(\mu-N,l+l')p$.
\end{proof}

\subsection{\texorpdfstring{Analysis of $\pi_{0,0}$}{Analysis of pi0,0}}\label{reduction}

Let us discuss the elements of the dual space $\widehat{G}$ obtained from the degenerate principal series, that is the equivalence classes of irreducible unitary subrepresentations of $\left\{\pi_{i\lambda,\delta}\,,\,\R\times\Z\right\}$. It is proved in \cite{Gross} that $\pi_{i\lambda,\delta}$ is irreducible if $(\lambda,\delta)\neq(0,0)$. Moreover, the Knapp-Stein operator $\widetilde{\mathcal{T}}_{i\lambda,\delta}$ (or its algebraic version $T_{i\lambda,\delta}$ in the non-standard model) exhibits a unitary equivalence between $\pi_{i\lambda,\delta}$ and $\pi_{-i\lambda,-\delta}$. It is also established in \cite{Gross} that $\pi_{0,0}$ splits into the direct sum of two irreducible subrepresentations.

The following result describes this splitting in terms of eigenspaces of the Knapp-Stein operators and specifies the $K$-module structure of the summands.

\begin{theorem}\label{Ktypethm}
The representation $\pi_{0,0}$ of $G$ is reducible and decomposes as \[\pi_{0,0}\simeq\pi_{0,0}^-\oplus\pi_{0,0}^+,\] where the irreducible summands $\pi_{0,0}^\pm$ are characterised by:

\begin{enumerate}
\item their $K$-type formula:
\[\pi_{0,0}^-\simeq\sideset{}{^\oplus}\sum_{l-l'\equiv2[4]}\Vll\qquad\text{and}\qquad\pi_{0,0}^+\simeq\sideset{}{^\oplus}\sum_{l-l'\equiv0[4]}\Vll;\]
\item the classical Knapp-Stein intertwiners: $\pi_{0,0}^+$ (resp. $\pi_{0,0}^-$) is the eigenspace for the eigenvalue $1$ (resp. $-1$) of $\widetilde{\mathcal{T}}_{0,0}$ acting on $V_{0,0}$.
\item the algebraic Knapp-Stein intertwiners: $\pi_{0,0}^+$ (resp. $\pi_{0,0}^-$) is the eigenspace for the eigenvalue $1$ (resp. $-1$) of $T_{0,0}$ acting on $L^2\left(\C^{2m+1}\right)$.
\end{enumerate}
\end{theorem}

\begin{proof}To establish \textit{(1)} and \textit{(2)}, we proceed as in the proof of Proposition \ref{spec}, except that no choice of a Weyl representative is required to identify $V_1^{l-l'}\otimes\C_\delta$ to $V_1^{l-l'}\otimes\C_{-\delta}$ when $\delta=0$. For $p\in\Vll\otimes\C_0\subset\Vll\otimes V_1^{l-l'}$ one has  $\widetilde{\mathcal{T}}_{0,0}\,p = B_{2N}(-N,l+l')(p.j)$ where $j$ acts by the representation $V_1^{l-l'}$ of $\mathrm{Sp}(1)$. Identifying $\mathrm{Sp}(1)$ to $\mathrm{SU}(2)$ so that $j$ is represented by $\left[\begin{array}{cc}0&1\\-1&0\end{array}\right]$, and $V_1^{l-l'}$ to $\Phi^{l-l'}[x,y]$, it appears that $j$ acts on the variables by $(x,y)\mapsto(y,-x)$ so that a 0-weight vector $\xi$ is send to $(-1)^{l-l'}\xi$ by $j$. It follows that the intertwining operator acts on every $K$-type $\Vll$ of $\pi_{0,0}$ by \[B_{2N}(-N,l+l')(-1)^{\frac{l-l'}{2}}.\] Since $B_{2N}(-N,l+l')=(-1)^{-\frac{l+l'}{2}}$ by definition, it follows that \[\widetilde{\mathcal{T}}_{0,0}\,p=(-1)^{-l'}p,\] hence the result. We postpone the proof of \textit{(3)} to the next paragraph where the algebraic Knapp-Stein operators are defined and studied.
\end{proof}

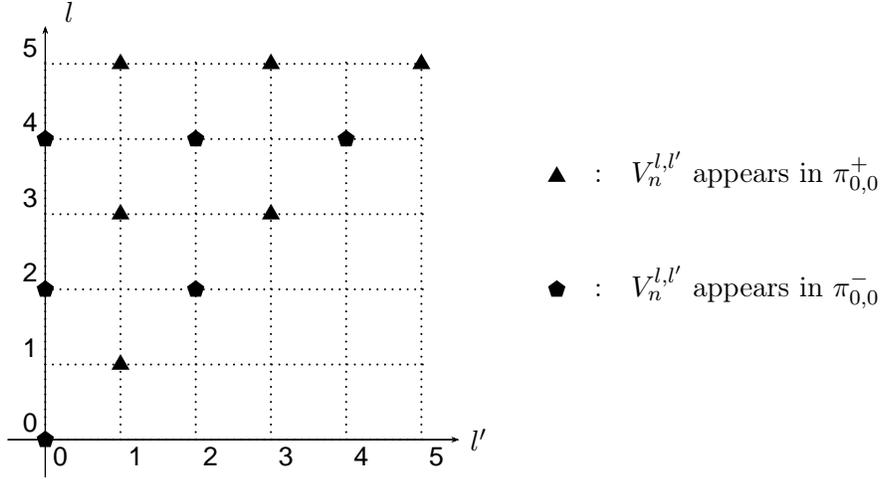
\begin{figure}[h!]
\begin{center}
\begin{pspicture*}(-1,-1)(11.2,6)
\psaxes[ticks=none, linewidth=0.01, labels=none]{->}(0,0)(-0.5,-0.5)(5.5,5.5)

\psdot[dotstyle=triangle*, dotsize=6pt](1,1)
\psdot[dotstyle=triangle*, dotsize=6pt](1,3)
\psdot[dotstyle=triangle*, dotsize=6pt](1,5)
\psdot[dotstyle=triangle*, dotsize=6pt](3,3)
\psdot[dotstyle=triangle*, dotsize=6pt](3,5)
\psdot[dotstyle=triangle*, dotsize=6pt](5,5)
\psdot[dotstyle=pentagon*, dotsize=6pt](0,0)
\psdot[dotstyle=pentagon*, dotsize=6pt](0,2)
\psdot[dotstyle=pentagon*, dotsize=6pt](0,4)
\psdot[dotstyle=pentagon*, dotsize=6pt](2,2)
\psdot[dotstyle=pentagon*, dotsize=6pt](2,4)
\psdot[dotstyle=pentagon*, dotsize=6pt](4,4)

\psgrid[subgriddiv=1,griddots=10](0,0)(5,5)

\psdot[dotstyle=triangle*, dotsize=6pt](6.8,3.5)
\uput{0.5}[r](6.8,3.55){:\quad $V_n^{l,l'}$ appears in $\pi_{0,0}^+$}
\psdot[dotstyle=pentagon*, dotsize=6pt](6.8,2.0)
\uput{0.5}[r](6.8,2.03){:\quad $V_n^{l,l'}$ appears in $\pi_{0,0}^-$}
\uput{0.5}[ur](-0.1,5.2){$l$}
\uput{0.65}[r](5,0){$l'$}
\end{pspicture*}

\end{center}
\caption{Repartition of the $K$-types of $\pi_{0,0}$}
\end{figure}

\section{A non-standard model and intertwining operators}\label{intertwiners}

This section is devoted to the description of a new model of degenerate principal series, in which intertwining operators happen to take an algebraic form.

\subsection{Non-standard model}

Let $\mu$ and $\delta$ be as above. The \emph{non-compact picture} described in Paragraph~\ref{noncompict} allowed to realise $\pi_{\mu,\delta}$ on the Hilbert space \[L^2\left(\HC\right)\simeq L^2\left(\C\times\C_1^m\times\C_2^m\right)\simeq L^2\left(\C^{2m+1}\right).\] Let $\mathcal{F}_{\C\times\C_2^m}$ be the partial Fourier transform defined on $L^2(\C\times\C_1^m\times\C_2^m)$ by \[\mathcal{F}_{\C\times\C_2^m}f(\tau,X_1,\xi_2)=\int_{\C\times\C_2^m}f(t,X_1,X_2)e^{-2i\pi\Re{\left(t\tau+\scal{X_2}{\xi_2}\right)}}\,dX_2\,dt.\]

\begin{definition}
The \emph{non-standard model} $\mathcal{U}_{\mu,\delta}$ of $\pi_{\mu,\delta}$ is the image of the non-compact picture $L^2(\HC)$ by \[\mathcal{F}_{\C\times\C_2^m}:L^2\left(\C^{2m+1}\right)\longrightarrow L^2\left(\C^{2m+1}\right),\] that is $\mathcal{U}_{\mu,\delta}=L^2\left(\C^{2m+1}\right)$ as a Hilbert space and the action of an element $g\in G$ on $\mathcal{U}_{\mu,\delta}$ is given by $\mathcal{F}_{\C\times\C_2^m}\circ\pi_{\mu,\delta}(g)\circ\mathcal{F}_{\C\times\C_2^m}^{-1}$.
\end{definition}

The equivalences between the induced, non-compact and non-standard models of $\pi_{\mu,\delta}$ are summed up in the following diagram:

\begin{equation}\label{alphamudelta}\alpha_{\mu,\delta}:\xymatrix @C=15mm @R=1mm {V_{\mu,\delta}\;\ar^-{\sim}_-{\mathrm{restrict.}}[r]&\;L^2(\HC)\;\ar^-{\sim}_-{\mathcal{F}_{\C\times\C_2^m}}[r]&\;\mathcal{U}_{\mu,\delta}\\
f\: \ar@{|->}[r] & \quad F\quad \ar@{|->}[r]&H}
\end{equation}

where, according to the embedding~(\ref{embed}), \[F(t,X_1,X_2)=f(1,X_1,2t,X_2)\] for $t\in\C$ and $(X_1,X_2)\in\C_1^m\times\C_2^m$.

\begin{lemma}\label{HFC}
Let $f\in V_{\mu,\delta}$. With notations as above, \[H(\tau,X_1,\xi_2)=\frac{1}{2}\mathcal{F}_{\C_2^n}f\left(1,X_1,\frac{\tau}{2},\xi_2\right).\]
\end{lemma}

\begin{proof}
According to the notations in the definition~(\ref{alphamudelta}) of $\alpha_{\mu,\delta}$, one has \begin{eqnarray*}H(\tau,X_1,\xi_2)&=&\int_{\C\times\C_2^m}F(t,X_1,X_2)e^{-2i\pi\Re{\left(t\tau+\scal{X_2}{\xi_2}\right)}}\,dX_2\,dt\\
&=&\int_{\C_2^n}f(1,X_1,2t,X_2)e^{-2i\pi\Re{\scal{\left(t,X_2\right)}{\left(\tau,\xi_2\right)}}}\,d(t,X_2)\\
&=&\frac{1}{2}\int_{\C_2^n}f(1,X_1,t,X_2)e^{-2i\pi\Re{\scal{\left(t,X_2\right)}{\left(\frac{\tau}{2},\xi_2\right)}}}\,d(t,X_2),
\end{eqnarray*}
hence the result.
\end{proof}

\subsection{Algebraic Knapp-Stein intertwiners}

We now come to the main point of this section, \textit{id est} the proof that the normalised Knapp-Stein operators considered in Section~\ref{FourierKnappStein} take an algebraic form once expressed in the non-standard model of the previous paragraph. More precisely, for $H\in L^2\left(\C^{2m+1}\right)$, we let \[T_{\mu,\delta}H(s,X_1,X_2)=\left|\frac{s}{2}\right|^{-\mu}\left[s\right]^{-\delta}H\left(s,\frac{2}{s}X_2,\frac{s}{2}X_1\right).\]
The jacobian determinant of the transform \[(s,X_1,X_2)\longmapsto\left(s,\frac{2}{s}X_2,\frac{s}{2}X_1\right)\] is easily seen to have modulus $1$, so that $T_{\mu,\delta}$ is an endomorphism of $L^2\left(\C^{2m+1}\right)$, which turns out to be the realisation of the normalised Knapp-Stein intertwiner $\widetilde{\mathcal{T}}_{\mu,\delta}$ in the non-standard picture:

\begin{theorem}\label{algKS}
For any $(\mu,\delta)\in i\R\times\Z$, the following diagram is commutative:
\[\xymatrix @C=15mm @R=15mm {V_{-\mu,-\delta}\ar^-{\widetilde{\mathcal{T}}_{\mu,\delta}}[r]\ar_{\mathcal{\alpha}_{-\mu,-\delta}}[d]&V_{\mu,\delta\ar^{\mathcal{\alpha}_{\mu,\delta}}[d]}\\\mathcal{U}_{-\mu,-\delta}\ar_-{T_{\mu,\delta}}[r]&\mathcal{U}_{\mu,\delta}}.\]
\end{theorem}

\begin{proof}Let $f\in V^\infty_{-\mu,-\delta}$. Then
\begin{align*}
\alpha_{\mu,\delta}\circ\widetilde{\mathcal{T}}_{\mu,\delta}\,f(\tau,X_1,\xi_2)&=\mathcal{F}_{\C\times\C_2^m}\mathcal{F}_{\mathrm{symp}}f(\tau,X_1,\xi_2)&\text{by Proposition~\ref{propFsympT},}\\
&=\frac{1}{2}\mathcal{F}_{\C_2^n}\mathcal{F}_{\text{symp}}\,f\left(1,X_1,\frac{\tau}{2},\xi_2\right)&\text{by Lemma~\ref{HFC},}\\
&=\frac{1}{2}\mathcal{F}_{\C_2^n}\,f\left(\frac{\tau}{2},\xi_2,1,X_1\right)&\text{by Lemma~\ref{flip},}\\
&=\frac{1}{2}\left|\frac{2}{\tau}\right|^\mu\left[\frac{2}{\tau}\right]^\delta\mathcal{F}_{\C_2^n}\,f\left(1,\frac{2}{\tau}\xi_2,\frac{\tau}{2},\frac{\tau}{2}X_1\right)&\text{by Lemma~\ref{lemmaFC2},}\\
&=\left|\frac{\tau}{2}\right|^{-\mu}\left[\tau\right]^{-\delta}\frac{1}{2}\mathcal{F}_{\C_2^n}\,f\left(1,\frac{2}{\tau}\xi_2,\frac{\tau}{2},\frac{\tau}{2}X_1\right)&\\
&=\left|\frac{\tau}{2}\right|^{-\mu}\left[\tau\right]^{-\delta}H\left(\tau,\frac{2}{\tau}\xi_2,\frac{\tau}{2}X_1\right)&\text{by Lemma~\ref{HFC},}
\end{align*}
thus proving that $\alpha_{\mu,\delta}\circ\widetilde{\mathcal{T}}_{\mu,\delta}=T_{\mu,\delta}\circ\alpha_{-\mu,-\delta}$
\end{proof}

We can now complete the proof of Theorem~\ref{Ktypethm}. Since $T_{0,0}^2=\mathrm{Id}_{L^2\left(\C^{2m+1}\right)}$, every function $H\in L^2\left(\C^{2m+1}\right)$ can be written in a unique way as $H=H_+ + H_-$ with $T_{0,0}H_+ = H_+$ and $T_{0,0}H_- = -H_-$. Indeed, \[H_\pm = \frac{1}{2}\left(H\pm T_{0,0}H\right),\] which gives the expected characterisation \textit{(3)} of $\pi_{0,0}^\pm$ in Theorem~\ref{Ktypethm}.

\subsection{Perspectives}

The existence of a \emph{non-standard} model for degenerate principal series of $\SpnR$ and $\SpnC$ in which the Knapp-Stein intertwiners take an algebraic form relies on the special form of the nilradical of the inducing parabolic subgroup. It is then natural to ask if the same occurs with parabolic subgroups of Heisenberg type in other groups.

\bigskip
\begin{center}
\textbf{Acknowledgements}
\end{center}
The results presented here were mostly obtained during the author's stay at the Gradutate School of Mathematical Sciences of the University of Tokyo. We wish to heartily thank Pr. Toshiyuki Kobayashi for his kindness as a host and many enlightening discussions during the preparation of this article. We also thank Pr. Pevzner for helpful discussions and the meetings he organised in Reims.

\bibliographystyle{amsplain}
\bibliography{biblio}
\end{document}